\documentclass{lms}
\usepackage[utf8]{inputenc}
\usepackage[american]{babel}
\usepackage[T1]{fontenc}
\usepackage{amssymb,amsmath,amsfonts}
\usepackage{algorithmic}
\usepackage{algorithm}
\usepackage{tikz}
\usepackage{nicefrac}
\usepackage{hyperref}
\usepackage{unicode}
\usepackage{stmaryrd}

\def\snote#1{}

\newtheorem{thm}{Theorem}[section]
\newtheorem{lem}[thm]{Lemma}

\newtheorem{prop}[thm]{Proposition}

\newnumbered{defi}[thm]{Definition}
\newnumbered{rem}{Remark}
\newnumbered{exe}{Example}
\newnumbered{prob}{Problem}

\makeatletter
\newenvironment{localblock}[1]{\@exmplstar{\relax}{#1}}{\@endexample}

\def\mat#1{\begin{pmatrix}#1\end{pmatrix}}
\def\smat#1{{\def\arraystretch{.7}\mat{#1}}}

\def\acco#1{\left\{#1\right\}}
\def\bcro#1{\left\llbracket#1\right\rrbracket}
\def\chev#1{\left\langle#1\right\rangle}
\def\cout#1{\mathsf{#1}}
\let\leq\leqslant \let\geq\geqslant
\def\sfdiv{\mathsf{divide}}

\newcommand{\F}{\mathbb{F}}

\newcommand{\tildO}{\tilde{O}}
\newcommand{\MM}{\cout{M}}

\newcommand{\RR}{\cout{R}}
\DeclareMathOperator{\loglog}{loglog}

\makeatletter
\def\full@line{8pt}
\def\doublefull@line{12pt}
\makeatother

\definecolor{purple}{rgb}{.49,.11,.61}
\definecolor{green}{rgb}{.08,.69,.10}
\definecolor{blue}{rgb}{.01,.26,.87}
\definecolor{pink}{rgb}{1,.5,.75}
\definecolor{brown}{rgb}{.39,.21,.00}
\definecolor{red}{rgb}{.89,0,0}
\definecolor{lightblue}{rgb}{.58,.81,.98} 
\definecolor{teal}{rgb}{0,.57,.52}
\definecolor{orange}{rgb}{.97,.45,.02}
\definecolor{lightgreen}{rgb}{.53,.99,.01}
\definecolor{magenta}{rgb}{.76,0,.47}
\hypersetup{colorlinks,linkcolor={red!70!black},
  citecolor={green!70!black},urlcolor={blue!70!black}}

\title[Explicit isogenies in any characteristic]{Explicit isogenies in quadratic time in any characteristic}
\author[L. De Feo, C. Hugounenq, J. Plût, É. Schost]%
  {Luca De Feo, Cyril Hugounenq, Jérôme Plût, and Éric Schost}

\classno{11Y40 (primary), 11G20, 14H52 (secondary)}

\extraline{This work was partially supported by the
  \href{http://www.digiteo.fr/}{DIGITEO} grant 2013-0531D (ARGC) and NSERC.}

\begin{document}
\maketitle

\begin{abstract}
  Consider two ordinary elliptic curves $E,E'$ defined over a finite field
  $\F_q$, and suppose that there exists an isogeny $\psi$ between $E$
  and $E'$.  We propose an algorithm that determines $\psi$ from the
  knowledge of $E$, $E'$ and of its degree $r$, by using the
  structure of the $ℓ$-torsion of the curves (where $ℓ$~is a prime
  different from the characteristic~$p$ of the base field).  

  Our approach is inspired by a previous algorithm due to Couveignes,
  that involved computations using the $p$-torsion on the curves. 
  The most refined version of that algorithm, due to De Feo, 
  has a complexity of~$\tildO(r^2) p^{O(1)}$ base field operations.
  On the other hand, the cost of our algorithm is $\tildO(r^2) \log(q)^{O(1)}$,
  for a large class of inputs; this makes it an interesting alternative
  for the medium- and large-characteristic cases.
%% The problem we will consider here is the computation of an isogeny between two elliptics curves with the knowledge of the domain and the codomain of the isogeny and it's degree $r$. Couveignes's algorithm is an algorithm which solves this problem in $O(r^2)$ operations using the $p$-torsion. We want to extend the method used by Couveignes  We try to adapt his method here to the case of the $2$ torsion and more generally to the $\ell$ torsion, thus we propose an alternative for medium characteristic with this algorithm.
\end{abstract}

% \section*{Proposed notation}

% This section is for internal reference only: erase after the paper has
% stabilized.

% \begin{itemize}
% \item $\mathbb{F}_q$ is the field we are working on
% \item $\ell$ is for the $\ell$ torsion we are working on
% \item $r$ is the degree of the isogeny we want to compute
% \item $k$ is the integer such that $\ell^{2k}>4r+1$
% \item we thus work with a tower which has for top level $F_{q^{\ell^k}}$
% \item $E$ is for ordinary elliptic curves defined over the finite field $\mathbb{F}_q$
% \item $\mathcal{O}$ (resp. $\mathcal{O}_x$) is the notation for the endomorphism ring associated (up to isomorphism) to $E$ (resp. $E_x$)
% \item $K$ is the notation for the imaginary quadratic field in which $\mathcal{O}$ is defined
% \item $d_K$ is the negative integer such that $K=\mathbb{Z}[d_K]$  
% \end{itemize}

%%%%%%%%%%%%%%%

\section{Introduction}
\label{sec:introduction}

Isogenies are non-zero morphisms of elliptic curves, that is,
non-constant rational maps preserving the identity element. They are
also algebraic group morphisms. Isogeny computations play a central
role in the algorithmic theory of elliptic curves. They are notably
used to speed up Schoof's point counting
algorithm~\cite{schoof85,atkin88,schoof95,elkies98}. They are
also widely applied in cryptography, where they are used to speed up
point multiplication~\cite{gallant+lambert+vanstone01,longa+sica14},
to perform cryptanalysis~\cite{mauer+menezes+teske01}, and to
construct new
cryptosystems~\cite{teske06,charles+lauter+goren09,Stol,defeo+jao+plut12,jao+soukharev2014-signatures}.

The \emph{degree} of an isogeny is its degree as a rational map. If an
isogeny has degree $r$, we call it an $r$-isogeny, and we say that two
elliptic curves are $r$-isogenous if there exists an $r$-isogeny
relating them. Accordingly, we say that two field elements $j$ and
$j'$ are $r$-isogenous if there exist $r$-isogenous elliptic curves
$E$ and $E'$ such that $j(E)=j$ and $j(E')=j'$. The
\emph{explicit isogeny} problem has many incarnations. In this paper,
we are interested in the variant defined below.

\begin{localblock}{Explicit isogeny problem} \label{prob:isogeny-problem}
  Given two $j$-invariants $j$ and $j'$, and a positive integer
  $r$, determine if they are $r$-isogenous. In that case, compute
  curves $E$, $E'$ with $j(E)=j$ and $j(E')=j'$, and the
  rational functions defining an $r$-isogeny $ψ:E\to E'$.
\end{localblock}

A good measure of the computational difficulty of the problem is given
by the isogeny degree~$r$. Indeed the output is represented by $O(r)$
base field elements, hence an asymptotically optimal algorithm would
solve the problem using $O(r)$ field operations. Even though the input
size is logarithmic in $r$, by a slight abuse we say that an algorithm
solves the isogeny problem in polynomial time if it does so in the
size of the output. Thanks to V\'elu's formulas~\cite{velu71}, in
particular the version appearing in~\cite[§2.4]{kohel}, we can compute
$ψ$ from the knowledge of the polynomial $h$ vanishing on the
abscissas of the points in $\kerψ$, at the cost of a constant number
of multiplications of polynomials of degree $O(r)$. Given that all
known algorithms to compute $h$ require more than a few polynomial
multiplications, we often say that we have computed $ψ$ whenever we
have computed $h$, and conversely.

This paper focuses on the explicit isogeny problem for \emph{ordinary}
elliptic curves over finite fields. A famous theorem by Tate~\cite{tate1966endomorphisms} states
that two curves are isogenous over a finite field if and only if they
have the same cardinality over that field. The explicit isogeny
problem stated here appears naturally in the Schoof-Elkies-Atkin point
counting algorithm (SEA). There, $E$ is a curve over $\F_q$, whose rational points
we wish to count, and $E'$ is an $r$-isogenous curve, with $r$
a prime of size approximately $\log(q)$. For this reason, the explicit
isogeny problem is customarily solved without prior knowledge of the
cardinality of $E(\F_q)$. We abide by this convention here.

Many algorithms
have been suggested over the years to solve the explicit isogeny
problem. Early algorithms were due to Atkin~\cite{atkin91} and
Charlap, Coley and
Robbins~\cite{charlap1991enumeration}. Elkies'~\cite{elkies98,Bostan}
was the first algorithm targeted to finite fields (of large enough
characteristic). Assuming $r$ is prime, its complexity is dominated by
the computation of the modular polynomial $\Phi_r$, which is an object
of bit size $O(r^3\log(r))$. Later Bröker, Lauter and
Sutherland~\cite{sutherland10:modpol} optimized the modular polynomial
computation in the context of the SEA
algorithm~\cite{sutherland2013evaluation}. Finally Lercier and Sirvent~\cite{lercier+sirvent08,1602.00244}
generalized Elkies' algorithm to work in any characteristic. Despite
these advances, the overall cost of Elkies' algorithm and its
variants is still at least cubic in $r$.

Another line of work to solve the explicit isogeny problem for ordinary curves was
initiated by Couveignes~\cite{couveignes94,couveignes96,couveignes00},
and later improved by De Feo and Schost~\cite{df10,df+schost12}. These
algorithms use an interpolation approach combined with ad-hoc
constructions for towers of finite fields of characteristic $p$. Their
complexity is quasi-quadratic in $r$, but exponential in $\log(p)$,
hence they are only practical for very small characteristic.

In this paper we present a variant of Couveignes' algorithm with
complexity polynomial in $\log(p)$ and quasi-quadratic in $r$. Like the original 
algorithm, it is limited to  isogenies of ordinary curves. Together
with the Lercier-Sirvent algorithm, they are the only polynomial-time
isogeny computation algorithms working in any characteristic, hence
they are especially relevant for counting points in \emph{medium}
characteristic (i.e., counting points over $\F_{p^n}$, when
$n\gg p/\log (p)$).

Note that, although Couveignes-type algorithms do not make use of the
modular polynomial $\Phi_r$, its computation is still necessary in the
context of the SEA algorithm. Thus our new algorithm does not improve
the overall complexity of point counting, though it may provide a 
speed-up in some cases. It
gives, however, an effective algorithm for solving the explicit
isogeny problem, with potential applications in other contexts, e.g.,
cryptography.

\subsection{Notation}

Throughout this paper: $r$~is a positive integer, $p$~an odd prime,
$q$~a power of $p$, and $\mathbb F_q$ is the finite field with
$q$~elements. $E$ ~is an ordinary elliptic curve over~$\mathbb F_q$,
its group of $n$-torsion points is denoted by~$E[n]$, its
$q$-Frobenius endomorphism by~$π$.  The endomorphism ring of $E$ is
denoted by~$\mathcal O$, with~$K = \mathcal O ⊗ ℚ$ the corresponding
number field, $\mathcal O_K$ its maximal order, and $d_K$~the
discriminant of~$\mathcal O_K$.
For a prime~$ℓ$ different from~$p$ and not dividing~$r$,
we denote by~$E[ℓ^k]$ the group of $ℓ^k$-torsion points of~$E$,
$E[ℓ^{∞}] = \varinjlim E[ℓ^k]$ the union of all $E[ℓ^k]$,
and $T_ℓ(E) = \varprojlim E[ℓ^k]$ the $ℓ$-adic Tate module~\cite[III.7]{Sil},
which is free of rank~two over~$ℤ_ℓ$.
The factorization of the characteristic polynomial of~$π$
over~$ℤ_ℓ$ is determined by the Kronecker symbol~$(d_K/ℓ)$.
If $(d_K/ℓ) = +1$ then we also define $λ,μ$ as
the eigenvalues of~$π$ in~$ℤ_ℓ$ and write~$h = v_ℓ(λ - μ)$,
where $v_ℓ$ is the $ℓ$-adic valuation.

We measure all computational complexities in terms of operations in
$\mathbb{F}_q$; the boolean costs associated to the algorithms
presented next are negligible compared to the algebraic costs, and
will be ignored. We use the Landau notation $O(\ )$ to express
asymptotic complexities, and the notation $\tildO(\ )$ to neglect
(poly)logarithmic factors.  We let $\MM(n)$ be a function such that
polynomials in $\F_q[x]$ of degree less than $n$ can be multiplied
using $\MM(n)$ operations in $\F_q$, under the assumptions
of~\cite[Chapter~8.3]{vzGG}. Using FFT multiplication, one can take
$\MM(n)∈ O(n\log(n)\loglog(n))$.

\subsection{Couveignes' algorithm and our contribution}
\label{sec:couv-algor}

Couveignes' isogeny algorithm takes as input two \emph{ordinary}
$j$-invariants $j,j'∈\F_q$, and a positive integer $r$ not
divisible by $p$, and returns, if it exists, an
$r$-isogeny $ψ:E\to E'$, with $j(E)=j$ and $j(E')=j'$. It is based on the observation that the
isogeny $ψ$ must put $E[p^k]$ in bijection with $E'[p^k]$, in a way
that is compatible with their structure as cyclic groups.
It proceeds in three steps:
\begin{enumerate}
\item\label{alg:orig-couveignes:tower} Compute generators $P,P'$ of
  $E[p^k]$ and $E'[p^k]$ respectively, for $k$ large enough;
\item\label{alg:orig-couveignes:interp} Compute the interpolation
  polynomial $L$ sending $x(P)$ to $x(P')$, and the abscissas of
  their scalar multiples accordingly;
\item\label{alg:orig-couveignes:rational} Deduce a rational fraction $g(x)/h(x)$
  that coincides with $L$ at all points of $E[p^k]$, and verify that
  it defines the $x$-component of an isogeny of degree $r$. If it
  does, return it; otherwise, replace $P'$ with a scalar multiple of
  itself and go back to Step~\ref{alg:orig-couveignes:interp}.
\end{enumerate}

For this algorithm to succeed, enough interpolation points are
required. Given that the $x$-component of $ψ$ is defined by $O(r)$
coefficients, we have $p^k∈\Theta(r)$. However, most of
the time, those points are not going to be defined in the base field
$\F_q$, so we must use efficient algorithms to construct and compute
in towers of extensions of finite fields. Indeed, Couveignes and his
successors go at great length in studying the arithmetic of
\emph{Artin-Schreier towers}~\cite{couveignes00,df+schost12}, and the
adaptation of the fast interpolation algorithm to that
setting~\cite{df10}.  Using these highly specialized constructions,
Steps~\ref{alg:orig-couveignes:tower}
and~\ref{alg:orig-couveignes:interp} are both executed in time
$\tildO(p^{k+O(1)})=\tildO(rp^{O(1)})$. However the last step only
succeeds for one pair of torsion points $P,P'$, in general, thus
$O(r)$ trials are expected on average.
Hence, the overall complexity of Couveignes' algorithm is
$\tildO(r^2p^{O(1)})$, i.e., quadratic in $r$, but exponential in
$\log(p)$. Although the exponent of $p$ is relatively small,
Couveignes algorithm quickly becomes impractical as $p$ grows.

In this paper we introduce a variant of Couveignes' algorithm with the
same quadratic complexity in $r$, and \textbf{no exponential
  dependency in $\log(p)$}.

The bottom line of our algorithm is elementary: replace $E[p^k]$ in
the algorithm with $E[ℓ^k]$, for some small prime $ℓ$. However a
naive application of this idea fails to yield a quadratic-time
algorithm. Indeed, in the worst case one has $ℓ^{2k}∈\Theta(r)$, with
$E[ℓ^k]≃(ℤ/ℓ^kℤ)^2$. Hence, two generators $P,Q$ of $E[ℓ^k]$ must
be mapped onto two generators of $E'[ℓ^k]$. This can be done in
$O(ℓ^{4k})$ possible ways, with a best possible cost of $O(\ell^{2k})$
per trial,
thus yielding an
algorithm of complexity $O(ℓ^{6k})=O(r³)$ at best.

To avoid this pitfall, we carefully study
in Section~\ref{sec:isogeny-volcanoes} the structure of
$E[ℓ^k]$, and its relationship with the Frobenius endomorphism $π$.
With that knowledge, we can put some restrictions on the generators $P,Q$,
as explained in Section~\ref{sec:acti-frob-endm},
thus limiting the number of trials to $O(ℓ^{2k})$.
In Section~\ref{sec:interpolation}
we present an interpolation algorithm adapted to the setting of
$ℓ$-adic towers, and in Section~\ref{sec:complete-algorithm} we put
all steps together and analyze the full algorithm. Finally in
Section~\ref{sec:implem} we discuss our implementation and the
performance of the algorithm.

%%%%%%%%%%%%%%%

\subsection{Towers of finite fields}
\label{sub:towers}

The algorithms presented next operate on elements defined in finite
extensions of $\F_q$. Specifically, we will work in a \emph{tower} of finite
fields $\F_q=F₀⊂F₁⊂\cdots⊂F_n$, with $\ell$ dividing $\#F_1-1$, $d_1=[F₁:F₀]$
dividing $ℓ-1$, and $[F_{i+1}:F_i]=ℓ$ for any $i>0$. For $ℓ=2$,
we build upon the work of Doliskani and Schost~\cite{DoSc12}, whereas for
general $ℓ$ we use towers of Kummer extensions in a way similar
to~\cite[\S2]{DeDoSc13}.  Both constructions represent elements of
$F_i$ as univariate polynomials with coefficients in $\F_q$, thus
basic arithmetic operations can be performed using modular
polynomial arithmetic over~$\F_q$. While constructing the tower, we also enforce
special relations between the generators of each level, so that moving
elements up and down the tower, and testing membership, can be done at
negligible cost.

We briefly sketch the construction for odd $\ell$. We first look for a
primitive polynomial $P_1∈\F_q[x]$ of degree equal to $[F₁:F₀]$. There
are many probabilistic algorithms to compute $P_1$ in expected time
polynomial in $\ell$ and $\log(q)$; since their cost does not depend
on the height $n$ of the tower, we neglect it (in all that follows, by
{\em expected cost} of an algorithm, we refer to a Las Vegas algorithm,
whose runtime is given in expectation). Then, the image $x₁$ of
$x$ in $F_1=\F_q[x]/P₁(x)$ is an element of multiplicative order
$\#F₁-1$, and in particular it is not a $ℓ$-th power. Hence for any
$i>1$ we define $F_i$ as $\F_q[x]/P_1\bigl(x^{ℓ^{i-1}}\bigr)$, the
computation of the polynomials $P_1\bigl(x^{ℓ^{i-1}}\bigr)$ incurring
no algebraic cost. Using this representation, elements of $F_i$ can be
expressed as elements of a higher level $F_{i+j}$, and reciprocally,
by a simple rearrangement of the coefficients. Another fundamental
operation can be done much more efficiently than in generic finite
fields, as the following generalization of~\cite[\S2.3]{DoSc12} shows.

\begin{lem}\label{lemma:frob-ell}
  Let $F_0⊂\cdots⊂F_n$ be a Kummer tower as defined above, and let
  $a∈F_i$ for some $0≤i≤n$. For any integer $j$, we can compute the
  $(\#F_j)$-th power of $a$ using $O(ℓ^{i-1}\MM(ℓ))$ operations in
  $\F_q$, after a precomputation independent of $a$ of cost
  $O(ℓ\MM(ℓ)\log(q))$.
\end{lem}
\begin{proof}
  Without loss of generality, we can assume that $j<i$; otherwise, the
  output is simply $a$ itself.  Let $s=\#F_j$, and let
  $d=[F_i:F_1]=ℓ^{i-1}$. Let $x_i$ be the image of $x$ in
  $F_i=\F_q[x]/P_i(x)$, so that $x_i^d=x_1$.

  The first step, independently of $a$, is to compute
  $y=x_i^s$. Writing $s = ud + r$, with $r<d$, we see that $y$ is
  given by $x₁^{u \bmod \#F₁}x_i^r$. We compute $x₁^{u\bmod\#F₁}$
  using $O(ℓ\MM(ℓ)\log(q))$ operations in $\F_q$, and we keep this
  element as a monomial of $F₁[x_i]$.
  By assumption, $a$ is represented as a polynomial in $x_i$ of degree
  less than $[F_i:F_0]$. We rewrite it as
  $a =a_0 + a_1 x_i + \cdots + a_{d-1} x_i^{d-1}$, with $a_i∈F₁$. This
  is done by a simple rearrangement of the coefficients of $a$.

  Finally, we compute $a(y)$ by a Horner scheme. All powers $y^k$ we
  need are themselves monomials in $F₁[x_i]$, each computed from the
  previous one using $O(\MM(ℓ))$ operations in $\F_q$, for a total of
  $O(ℓ^{i-1}\MM(ℓ))$. Finally the monomials $a_ky^k$ are combined
  together to form a polynomial in $(x_1,x_i)$ of degree less than
  $(d_1,d)$, and then brought to a canonical form in $F_i$ via another
  rearrangement of coefficients.
\end{proof}

Summarizing, the following computations can be performed
in a Kummer tower at the indicated asymptotic costs, all expressed in 
terms of operations in $\F_q$.
\begin{itemize}
\item basic arithmetic operations (addition, multiplication) in $F_i$,
  using $O(\MM(ℓ^i))$ operations;
\item inversion in $F_i$ using $O(\MM(ℓ^i)\log(ℓ^i))$
  operations (when $ℓ=2$, a factor of $i$ can be saved
    here~\cite{DoSc12}, but we will disregard this optimization for
    simplicity.)
\item mapping elements from $F_{i-1}$ to $F_i$ and \emph{vice versa}
  at no arithmetic cost;
\item multiplication and Euclidean division of polynomials of degree
  at most $d$ in $F_i[x]$ using $O(\MM(dℓ^i))$ operations, via
  Kronecker's
  substitution, as already done in e.g.~\cite{vzgathen+shoup92};
\item computing a $(\#F_j)$-th power in $F_i$ using $O(ℓ^{i-1}\MM(ℓ))$
  operations, after a precomputation that uses $O(ℓ\MM(ℓ)\log(q))$
  operations.
\end{itemize}

For one fundamental operation, we only have an efficient algorithm in
the case $ℓ=2$, hence we introduce the following notation:
\begin{itemize}
\item $\RR(i)$ is a bound on the expected cost of finding a root of a polynomial of degree
  $ℓ$ in $F_i[x]$.
\end{itemize}
Note that we allow Las Vegas algorithms here, as no deterministic polynomial
time algorithm is known. For $\ell=2$, Doliskani and Schost show that
$\RR(i)=O(\MM(ℓ^i)\log(ℓ^iq))$. For general $ℓ$, we have
$\RR(i)=O(ℓ^i\MM(ℓ^{i+1})\log(ℓ)\log(ℓq))$ using the variant of the
Cantor-Zassenhaus algorithm described in~\cite[Chapter~14.5]{vzGG}, or
$\RR(i)=O\bigl((ℓ^{i(ω+1)/2}+\MM(ℓ^{i+1}\log (q)))i\log(ℓ)\bigr)$
using~\cite{kaltofen+shoup97}. Here, $\omega$ is such that matrix
multiplication in size $m$ over any ring can be done in $O(m^\omega)$
base ring operations (so we can take $\omega =2.38$ using the 
Coppersmith-Winograd algorithm). In any case, $\RR(i)$ is between 
linear and quadratic in the degree $\ell^{i}$.

\section{The Frobenius and the volcano}
\label{sec:isogeny-volcanoes}

In this section we explore some fundamental properties of ordinary
elliptic curves over finite fields: the structure of their isogeny
classes, its relationship with the rational $ℓ^∞$-torsion points, and
with the Frobenius endomorphism $π$.

\subsection{Isogeny volcanoes}

For an extensive introduction to isogeny volcanoes we refer the
reader to~\cite{sutherland2013isogeny}.  We recall here, without their
proof, two results about $ℓ$-isogenies between ordinary elliptic
curves.

\begin{prop}[{\cite[Proposition~21]{kohel}}] \label{prop:isogeny-asc-desc}
Let~$ϕ: E → E'$ be an $ℓ$-isogeny between ordinary elliptic curves
and~$\mathcal O, \mathcal O'$ be their endomorphism rings.
Then one of the three following cases is true:
\begin{enumerate}
\item $[\mathcal O':\mathcal O] = ℓ$,
in which case we call $ϕ$ \emph{ascending};
\item $[\mathcal O:\mathcal O'] = ℓ$,
in which case we call $ϕ$ \emph{descending};
\item $\mathcal O' = \mathcal O$,
in which case we call $ϕ$ \emph{horizontal}.
\end{enumerate}
\end{prop}
\begin{prop}[{\cite[Proposition~23]{kohel}; \cite[Lemma~6]{sutherland2013isogeny}}] \label{prop:isogeny-count}
Let~$E$ be an ordinary elliptic curve with endomorphism ring~$\mathcal O$.
\begin{enumerate}
\item If $\mathcal O$~is $ℓ$-maximal then
there are $(d_K/ℓ)+1$~horizontal $ℓ$-isogenies from~$E$
(and no ascending $ℓ$-isogenies).
\item If $\mathcal O$~is not $ℓ$-maximal then
there are no horizontal $ℓ$-isogenies from~$E$,
and one ascending $ℓ$-isogeny.
\end{enumerate}
\end{prop}

A \emph{volcano of $ℓ$-isogenies} is a connected component
of the graph of rational $ℓ$-isogenies between curves defined on~$\mathbb F_q$.
The \emph{crater} is the subgraph corresponding to curves
having an $ℓ$-maximal endomorphism ring.
The shape of the crater is given by the Kronecker symbol~$(d_K/ℓ)$,
as per Proposition~\ref{prop:isogeny-count}.
For any~$k ≥ 0$, an $ℓ^k$-isogeny is \emph{horizontal}
if it is the composite of $k$~horizontal $ℓ$-isogenies.
The \emph{depth} of a curve is its distance from the crater.
It is also the $ℓ$-adic valuation of the conductor
of~$\mathcal O = \mathrm{End}(E)$.

		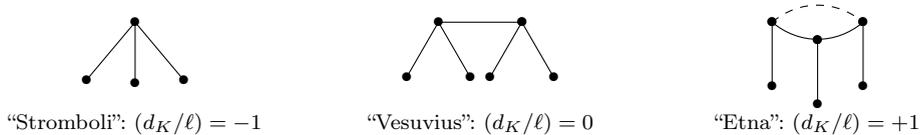
\begin{figure}[h]
		\begin{center}
        \begin{tikzpicture}[scale=0.2]
        \coordinate (A) at (0,0);
		\coordinate (B) at (230:5);
		\coordinate (C) at (270:4);
		\coordinate (D) at (310:5);
		\draw (A) node{$\bullet$};
		\draw (B) node{$\bullet$};
		\draw (C) node{$\bullet$};
		\draw (D) node{$\bullet$};
		\node at (0,-6.7) {``Stromboli'': $(d_K/ℓ) = -1$};
		\draw (A)--(B);
		\draw (A)--(C);
		\draw (A)--(D);
		
		\begin{scope}[xshift=20cm]
		\coordinate (A) at (0,0);
		\coordinate (B) at (5.5,0);
		\coordinate (C) at (240:4.2);
		\coordinate (D) at (300:4.2);
		\draw (A) node{$\bullet$};
		\draw (B) node{$\bullet$};
		\draw (C) node{$\bullet$};
		\draw (D) node{$\bullet$};
		\node at (2.8,-6.7) {``Vesuvius'': $(d_K/ℓ) = 0$};
		\draw (A)--(B);
		\draw (A)--(C);
		\draw (A)--(D);
		\end{scope}
		
		\begin{scope}[xshift=25.5cm]
		\coordinate (A) at (0,0);
		\coordinate (C) at (240:4.2);
		\coordinate (D) at (300:4.2);
		\draw (A) node{$\bullet$};
		\draw (C) node{$\bullet$};
		\draw (D) node{$\bullet$};
		\draw (A)--(C);
		\draw (A)--(D);
		\end{scope}
		
		\begin{scope}[xshift=45cm]
		\node (A) at (-3,0) {$\bullet$};
		\node (B) at (3,0) {$\bullet$};
		\node (C) at (270:1.2) {$\bullet$};
		\node (D) at (90:1.5) {};
		\node at (0,-6.7) {``Etna'': $(d_K/ℓ) = +1$};
		%\draw[-] (A.center) to[bend right=25] (C.center);
		\draw[-,dashed] (A.center) to[bend left=40] (B.center);
		%\draw[-] (B.center) to[bend left=25] (C.center);
		%\draw[-,dashed] (B.center) to[bend right] (D.center);
		\draw[-] (A.center) to[bend right=40] (B.center);
			\begin{scope}[xshift=-3cm]
			\coordinate (A) at (0,0);
			\coordinate (C) at (270:4.2);
			\draw (A) node{$\bullet$};
			\draw (C) node{$\bullet$};
			\draw (A)--(C);
			\end{scope}
			\begin{scope}[xshift=3cm]
			\coordinate (A) at (0,0);
			\coordinate (C) at (270:4.2);
			\draw (A) node{$\bullet$};
			\draw (C) node{$\bullet$};
			\draw (A)--(C);
			\end{scope}
			\begin{scope}[yshift=-1.2cm]
			\coordinate (A) at (0,0);
			\coordinate (C) at (270:4.2);
			\draw (A) node{$\bullet$};
			\draw (C) node{$\bullet$};
			\draw (A)--(C);
			\end{scope}
		\end{scope}
		\end{tikzpicture}
		%\end{center}
		\caption{The three shapes of volcanoes of $2$-isogenies }
		\end{center} 
		\end{figure}
\kern -2ex
\subsection{The $ℓ$-adic Frobenius}

In the rest of this paper we consider only a volcano with a cyclic
crater (i.e. we assume $(d_K/\ell) = +1$),
so that $ℓ$~is an Elkies prime for these curves.
This implies that the Frobenius automorphism on~$T_ℓ(E)$,
which we write~$π|T_ℓ(E)$, has two distinct eigenvalues~$λ ≠ μ$.
The depth of the volcano of $\F_q$-rational $ℓ$-isogenies
is~$h = v_ℓ(λ-μ)$~\cite[Theorem 7(iv)]{sutherland2013isogeny}.

\begin{prop}\label{prop:matrice-frobenius}
Let~$E$ be an ordinary elliptic curve with Frobenius endomorphism~$π$.
Assume that the characteristic polynomial of~$π$
has two distinct roots~$λ, μ$ in~$ℤ_ℓ$,
so that the $ℓ$-isogeny volcano has a cyclic crater.
Then there exists a unique $e ∈ \llbracket 0, h\rrbracket$
such that $π|T_ℓ(E)$~is conjugate, over~$ℤ_ℓ$,
to the matrix $\left ( \begin{smallmatrix}λ & ℓ^e \\ 0 & μ
\end{smallmatrix}\right )$.
% where $a ∈ ℤ$, $0 ≤ a ≤ ℓ^{h} - 1$,
Moreover $e = h$ if~$E$ lies on the crater,
and else $h - e$~is the depth of~$E$ in the volcano.
\end{prop}

We note here that the matrix $\left(\begin{smallmatrix} λ & ℓ^h \\ 0 &
μ \end{smallmatrix}\right)$ is conjugate over~$ℤ_ℓ$
to $\left(\begin{smallmatrix} λ & 0 \\ 0 & μ\end{smallmatrix}\right)$.
\begin{proof}
Since the characteristic polynomial of~$π$ splits over~$ℤ_ℓ$,
the matrix of~$π|T_ℓ(E)$ is trigonalizable.
Conjugating the matrix $\left ( \begin{smallmatrix}λ & a\\0 & μ
\end{smallmatrix}\right )$ by~$\left ( \begin{smallmatrix}1 & b\\0 & 1
\end{smallmatrix} \right )$ replaces~$a$ by~$a - b (λ - μ)$,
and conjugating by~$\left(\begin{smallmatrix} c & 0 \\ 0 &
1\end{smallmatrix}\right)$ replaces~$a$ by~$c · a$,
so that the valuation~$e = v_ℓ(a)$ is an invariant under matrix conjugation.
This proves the first part.
For the second part, by Tate's theorem~\cite[Isogeny theorem 7.7 (a)]{Sil},
$\mathcal O ⊗ ℤ_ℓ$~is isomorphic to the order in~$ℚ_ℓ[π_ℓ]$
of matrices with integer coefficients,
which is generated by the identity and~$ℓ^{-\min (h, v_ℓ(a))} (π_ℓ-λ)$.
\end{proof}

We now study the action of $ℓ$-isogenies on the $ℓ$-adic Frobenius by
showing the link between two related notions of diagonalization.

\begin{defi}[(Horizontal and diagonal bases)]
  Let~$E$ be a curve lying on the crater. We call a point of~$E[ℓ^k]$
  \emph{horizontal} if it generates the kernel of a horizontal
  $ℓ^k$-isogeny.  We call a basis of~$E[ℓ^k]$ \emph{diagonal} if
  $π$~is diagonal in it, \emph{horizontal} if both basis points are
  horizontal.
\end{defi}

\begin{prop} \label{prop:diagonal-horizontal}
Let~$E$ be a curve lying on the crater and~$P$ be a point of~$E[ℓ^k]$
such that $ℓ^h P$~is an eigenvector of~$π$.
Then $ℓ^h P$~is horizontal if, and only if, $P$~is an eigenvector for~$π$.
If $π(P) = λ P$ then we say that $ℓ^h P$~has direction~$λ$.
\end{prop}

This proposition being trivially true for~$h ≥ k$,
we assume that~$k ≥ h$ in what follows.

Let $R$ be a point of $E$ of order~$ℓ^k$, let $ϕ$ be the isogeny 
with kernel~$\chev{R}$, and let $E'$ be its image. The subgroup~$\chev{R}$ defines a point in
the projective space of~$E[ℓ^k]$,
which is a projective line over~$ℤ/ℓ^k ℤ$.
There exists a canonical bijection~\cite[II.1.1]{SL2} between
this projective line and
the set of lattices of index~$ℓ^k$ in the $ℤ_ℓ$-module $T_ℓ(E)$:
it maps a line~$\chev{R}$ to the lattice~$Λ_R = \chev{R} + ℓ^k T_ℓ(E)$.
This lattice is also the preimage by~$ϕ$
of the lattice~$ℓ^k T_ℓ(E')$.

Fix a basis~$(P, Q)$ of~$E[ℓ^k]$, let $Π$ be the matrix of $π$
in this basis, and let~$R = x P + y Q$.
The lattice~$Λ_R$ is generated by the columns of the matrix
%% $L_R = \smat{ℓ^k & 0 & x\\0 & ℓ^k & y}$.
$L_R = \left (\begin{smallmatrix}ℓ^k & 0 & x\\0 & ℓ^k & y\end{smallmatrix} \right )$.
The Hermite normal form of~$L_R$
is 
%% $M_R = \smat{ℓ^{k-m} & x/y' \\ 0 & ℓ^m}$,
$M_R = \left (\begin{smallmatrix}ℓ^{k-m} & x/y' \\ 0 & ℓ^m\end{smallmatrix}\right )$,
where we write~$y = ℓ^m y'$ with~$ℓ ∤y'$,
and the columns of $M_R$ also generate the lattice~$Λ_R$.
We check that $M_R$ has determinant~$ℓ^k$.
Since $Λ_R = ϕ_R^{-1} (ℓ^k T_{ℓ} (E'))$,
there exists a basis of~$T_ℓ(E')$
in which $ϕ_R$ has matrix~$ℓ^k M_R^{-1}$.
Therefore, in that basis of~$T_ℓ(E')$,
the matrix of~$π|T_ℓ(E')$ is $M_R^{-1} · Π · M_R^{}$.
\begin{proof}[of Proposition~\ref{prop:diagonal-horizontal}]
Fix a basis~$(R, S)$ of~$E[ℓ^k]$ that diagonalizes~$π$.
We can write $P = x R + y S$;
without loss of generality we may assume $y=1$.
Let~$ϕ$ be the isogeny determined by~$ℓ^h P$, and let~$E'$ be its image.
Since $ℓ^h P$~is an eigenvector of~$π$, $ϕ$~is a rational isogeny.
According to the previous discussion,
$π|T_ℓ(E')$ has matrix~$\left ( \begin{smallmatrix}λ& ℓ^{h-k} x (λ-μ)\\ 0&μ
\end{smallmatrix}\right )$.
This matrix is diagonalizable only if~$v_{ℓ}(x) ≥ k - h$.
On the other hand, we can compute~$(π - μ) P = x (λ - μ) R$,
so that $P$~is an eigenvector on the same condition~$v_{ℓ}(x) ≥ k-h$.
\end{proof}

While horizontal bases are our main interest,
diagonal bases are easier to compute in practice.
Algorithms computing both kind of bases
are given in Section~\ref{sec:acti-frob-endm}.
The main tool for this is the next proposition:
given a horizontal point of order~$ℓ^k$,
it allows us to compute a horizontal point of order~$ℓ^{k+1}$.

\begin{prop}\label{prop:push-horizontal}
Let~$ψ: E → E'$ be a horizontal $ℓ$-isogeny with direction~$λ$.
For any point~$Q ∈ E[ℓ^∞]$,
if $ℓ Q$~is horizontal with direction~$μ$,
then $ψ(Q)$ is horizontal with direction~$μ$.
\end{prop}
\begin{proof}
Let~$Q' = ψ(Q)$ and~$\widehat{ψ}$ be the isogeny dual to~$ψ$.
Since both $\widehat{ψ}$~and~$\widehat{ψ}(Q') = ℓ Q$ are horizontal
with direction~$μ$, $Q'$~is also horizontal.
\end{proof}
\begin{prop}\label{prop:parallel}
Let~$ψ: E → E'$ be an isogeny of degree~$r$ prime to~$ℓ$.
\begin{enumerate}
\item The curves~$E$ and~$E'$ have the same depth
in their $ℓ$-isogeny volcanoes.
\item\label{prop:parallel:func} For any point~$P ∈ E[ℓ^k]$,
the isogenies with kernel $\chev{P}$ and~$\chev{ψ(P)}$ have the same type
(ascending, descending, or horizontal with the same direction).
\item\label{prop:parallel:ascent} If $P ∈ E[ℓ]$ and $P' ∈ E'[ℓ]$ are both ascending,
or both horizontal with the same direction,
then $E/P$ and~$E'/P'$ are again $r$-isogenous.
\end{enumerate}
\end{prop}
\begin{proof}
Points~(i) and~(ii) are consequences of Proposition~\ref{prop:matrice-frobenius}
and of the fact that $ψ$, being rational and of degree prime to~$ℓ$,
induces an isomorphism between the Tate modules of~$E$ and~$E'$,
commuting to the Frobenius endomorphisms.
For point~(iii), we just note that
since there exists a unique subgroup of order~$ℓ$ which is
either ascending or horizontal with a given direction,
we must have~$\chev{P'} = \chev{ψ(P)}$.
\end{proof}

%%%%%%%%%%%%%%%
\subsection{Galois classes in the $ℓ$-torsion}
Assume that $E$~has a $ℓ$-maximal endomorphism ring.  The
following proposition summarizes the properties of $E[\ell^k]$ that we
will need for our main interpolation algorithm.  If $ℓ$~is odd, let~$α
= v_ℓ(λ^{ℓ-1}-1)$ and~$β=v_ℓ(μ^{ℓ-1}-1)$; if $ℓ=2$,
let~$α=v_2(λ^2-1)-1$ and~$β = v_2(μ^2-1)-1$, and assume without loss
of generality that $α ≥ β$.  Since $λ ≢ μ \pmod{ℓ^{h+1}}$, it is
impossible that $λ ≡ μ ≡ 1 \pmod{ℓ^h}$, so that one at least of the two
valuations~$α, β$ is~$≤ h$, and therefore~$β ≤ h$.
\label{sub:classes}
\begin{prop}\label{prop:classes}
For any~$k$, let~$d_k$ be the degree of the smallest field extension $F/\F_q$
such that~$E[ℓ^k]⊂E(F)$. Then:
\begin{enumerate}
\item The order of $q$ in $(ℤ/ℓℤ)^×$ divides $d_1$,
and $d_1$ divides~$(ℓ-1)$.
\item If $ℓ$~is odd then for all $k ≥ 1$,
$ d_k = ℓ^{\min (v_ℓ (d_1), k - β)}$.
% $d_k = \mathrm {lcm} (d_1, ℓ^{k-β})$.
\item If $ℓ=2$ then $d_2 ∈ \acco{1,2}$ and, for all~$k ≥ 2$,
% $d_k = \mathrm{lcm}(d_2, ℓ^{k-β})$.
$d_k = ℓ^{\min (v_ℓ (d_2), k - β)}$.
\item Let $[F:\F_q]=d_1ℓ^n$, the group $E[ℓ^{∞}](F)$ is isomorphic to~$(ℤ/ℓ^{n+α} ℤ) × (ℤ/ℓ^{n+β} ℤ)$.
\item\label{prop:classes:count} The group~$E[ℓ^k]$ contains at most
$k · ℓ^{k+β}$ Galois conjugacy classes over~$F_1 = \F_{q^{d_1}}$.
\end{enumerate}
\end{prop}
\begin{proof}
The degree~$d_k$ is exactly the order of the matrix~$π|E[ℓ^k]$.
It is therefore the least common multiple of the multiplicative orders
of~$λ, μ$ modulo~$ℓ^k$.
This proves~(i) using the fact that~$λ · μ = q$.
% We conclude using the multiplicative structure of~$(ℤ/ℓ^k ℤ)^×$, and the fact
% that $λμ=q$.
% Note that in general the order of~$q$ is a strict divisor of~$d_1$,
% as is for example the case for~$π^2 - π + 29 = 0$ and~$ℓ = 7$,
% where $q = 29 ≡ 1 \pmod{7}$ and $d_1 = 6$.
For points~(ii)--(v) we may assume that $d_1 = 1$.
Then, for any~$N$, $v_ℓ(λ^{2N}-1) = α + v_{ℓ} (2N)$.
Let~$(P, Q)$ be a diagonal basis of~$E[ℓ^k]$.
The point $(π^N - 1) (x P + y Q) = (λ^N-1) x P + (μ^N-1) y Q$
is zero iff $v_{ℓ} (x) + α + v_{ℓ} (N) ≥ k$
and~$v_{ℓ} (y) + β + v_{ℓ} (N) ≥ k$. This shows~(iv).
The largest Galois classes
are those for which~$v_{ℓ} (y) = 0$ and their size is~$ℓ^{k - β}$,
proving~(ii) and~(iii).
Moreover, for any~$i ≤ k-β$ the points in an orbit of size~$≤ ℓ^i$
are those for which~$v_{ℓ} (x) ≥ k - α - i$ and~$v_{ℓ} (y) ≥ k - β - i$;
there are at most $ℓ^{\min(α+i, k) + \min (β+i, k)}$ such points,
and therefore $ℓ^{\min(α+i, k) + \min(β, k-i)} ≤ ℓ^{k-i+β}$
corresponding classes.
Summing this over all~$i$ proves~(v).
\end{proof}
\section{Computing the action of the Frobenius endomorphism}
\label{sec:acti-frob-endm}

We continue here our study on the action of the Frobenius $π$ on
$E[ℓ^k]$.  Given an ordinary elliptic curve~$E$ with $ℓ$-maximal endomorphism
ring, we explicitly compute diagonal and horizontal bases of $E[ℓ^k]$
as defined in the previous section.  We will use the latter basis of
$E[ℓ^k]$ in Section~\ref{sub:final-interp}, to put restrictions on the
interpolation problem of our algorithm.

We suppose that $k \geqslant h$. By Proposition~\ref{prop:classes}, there exists a Kummer tower
$F₀⊂\cdots⊂F_{k-\beta}$ such that all the points of~$E[ℓ^k]$ are
rational over~$F_{k-\beta}$. The algorithms presented next assume that
the tower has already been computed.

\subsection{Computation of a diagonal basis}
\label{ss:diagonal}

In Algorithm~\ref{alg:diagonal} below, we describe how to compute
eigenvalues of the Frobenius $\bmod \ell^k$ and corresponding
eigenvectors in the $\ell^{k}$-torsion subgroup.  We
write~$Q ← \sfdiv(ℓ, P)$ for the computation of a preimage of~$P$ by
multiplication by~$ℓ$.

\begin{algorithm}
\caption{\label{alg:diagonal}Computing a diagonal basis of $E[ℓ^k]$}
\begin{algorithmic}[1]
\REQUIRE $E$: an ordinary, $ℓ$-maximal elliptic curve;
$k$: a positive integer;
\ENSURE $(P_k, Q_k )$: a basis of $E[\ell^k]$;
$λ, μ ∈ ℤ/ℓ^k ℤ$
such that $\pi(P_k)= λ P_k$, $ \pi(Q_k)= μ Q_k$.
% \STATE Compute $(P_1,Q_1)$, a basis of $E[ℓ]$
\STATE{$λ ← 0$; $μ ← 0$; $P_0, Q_0 ← $ neutral element of~$E[ℓ]$.}
\FOR {$i=0$  to  $k-1$}
\STATE\label{alg:diagonal:divide}
  $P' ← \sfdiv(\ell, P_{i})$; $Q' ← \sfdiv (\ell, Q_{i})$.
\STATE\label{alg:diagonal:frobenius}
  compute $\pi|(P',Q')=\left( \begin{smallmatrix}
λ + aℓ^{i} & bℓ^{i}\\
cℓ^{i} & μ + dℓ^{i} \end{smallmatrix} \right) \pmod {ℓ^{i+1}}.$
\STATE\label{alg:diagonal:solve1}
  \textbf{if} $b = 0$ \textbf{then} $x ← 0$;
  solve equation~$c ℓ^{i} + ((d-a) ℓ^{i} + μ-λ) y = 0$;
\STATE\label{alg:diagonal:solve2}
  \textbf{else} solve equation
  $c ℓ^{i} x^2 + ((d-a) ℓ^{i}+ μ-λ) x - b ℓ^{i} = 0$;
  $y ← -cx/b$; \textbf{end if}.
\STATE\label{alg:diagonal:upd-P}
  $P_{i+1} ← P' + y Q'$; $Q_{i+1} ← x P' + Q'$.
\STATE{$λ ← λ + ℓ^{i} (a + bx)$; $μ ← μ + ℓ^{i} (d + cy)$.}
\ENDFOR
\RETURN $(P_{k},Q_{k},λ,μ).$
\end{algorithmic}
\end{algorithm}

%% \kern -4ex
% \begin{algorithm}%<<<
% \caption{\label{alg:diagonal}Computing a diagonal basis of $E[ℓ^k]$}
% \begin{algorithmic}[1]
% \REQUIRE $E$: an ordinary, $ℓ$-maximal elliptic curve;
% $k$: an integer~$≥ h$.
% \ENSURE $(P_k, Q_k )$: a basis of $E[\ell^k]$;
% $λ, μ ∈ ℤ/ℓ^k ℤ$
% such that $\pi(P_k)= λ P_k$, $ \pi(Q_k)= μ Q_k$.
% \STATE Compute $(P_1,Q_1)$, a basis of $E[ℓ]$
% \STATE $h':=1, u:=1$
% \FOR {$i=1$  to  $k-1$}
% \STATE\label{alg:diagonal:divide} $P' \leftarrow \sfdiv(\ell, P_i)$; $Q' \leftarrow \sfdiv(\ell, Q_i)$.
% \STATE\label{alg:diagonal:frobenius} compute $\pi|(P',Q')=\left( \begin{smallmatrix}
% λ + a\ell^{i} & b\ell^{i}\\
% c\ell^{i} & μ + d\ell^{i}
% \end{smallmatrix}
% \right) \pmod {\ell^{i+1}}$
% with $a,b,c,d \in \mathbb{Z}/\ell\mathbb{Z}$
% \STATE \textbf{if} {$λ \neq μ$} \textbf{then}
% $u \leftarrow (λ -μ)/\ell^{h'}$ \textbf{endif}
% \STATE $(λ, μ) \gets
%   (λ + a\ell^i, μ + d\ell^i)$
% \STATE $(b',c') \gets (-b/u , c/u) \pmod{ℓ}$
% \STATE $(P_{i+1},Q_{i+1}) \gets
%   (P'+\ell^{i-1-h'}b' Q',\;Q'+\ell^{i-1-h'}c' P')$
% \STATE \textbf{if} {$λ = μ$} \textbf{then} $h' \leftarrow h'+1$ \textbf{endif}
% \ENDFOR
% \RETURN $(P_{k},Q_{k},λ,μ)$
% \end{algorithmic}
% \end{algorithm}%>>>
\begin{prop}\label{th:diagonal}
  Algorithm~\ref{alg:diagonal} computes a diagonal basis of~$E[ℓ^k]$
  using an expected
  $O(\RR(k-\beta) + ℓ^2\MM(ℓ^{k-β}) + ℓ\MM(\ell^2)\log(\ell)\log(\ell
  q))$ operations in $\F_q$.
\end{prop}
\begin{proof}
The equation at line~\ref{alg:diagonal:solve1} or~\ref{alg:diagonal:solve2}
is first divided out by the largest power of~$ℓ$ possible,
which is~$ℓ^{\min (h, i)}$, then solved modulo $ℓ$.
For~$i ≤ h-1$, since~$a = d$ and~$b = c = 0$,
the solutions are~$x = y = 0$, and
steps~\ref{alg:diagonal:solve1} to~\ref{alg:diagonal:upd-P} do nothing.
A straightforward calculation shows that after each loop the basis
$(P_{i+1},Q_{i+1})$ is diagonal.

For~$i = 0$, the basis of~$E[ℓ](F_1)$ at step~\ref{alg:diagonal:divide}
is computed by factoring the $ℓ$-division polynomial
at an expected cost of $O(ℓ\MM(\ell^2)\log(\ell)\log(\ell q))$
operations using the Cantor-Zassenhaus algorithm.
  Once $E[ℓ]$ has been computed, we can factor the
  multiplication-by-$ℓ$ map as a product of two $ℓ$-isogenies. Then,
  for any $P$ defined in $E(F_{i-β})$, the computation of
  $\sfdiv(ℓ, P)$ at Step~\ref{alg:diagonal:divide} costs $O(\RR(i-β+1))$
  operations.
  Evaluating~$π(P')$ in Step~\ref{alg:diagonal:frobenius} has a cost
  of~$O(\ell^{i-\beta}\MM(\ell))$.
  Writing~$π(P')$ as a linear combination~$α P' + β Q'$ needs at
  most~$ℓ^2$ point additions, with a cost of~$ℓ^2
  \mathsf{M}(ℓ^{i-\beta+1})$.
  The cost of solving the equations at Steps~\ref{alg:diagonal:solve1}
  and~\ref{alg:diagonal:solve2} by exhaustive search is negligible, as
  are the remaining operations.  Since the cost of each loop grows
  geometrically, the last loop dominates all others, and gives the
  stated complexity.
\end{proof}

\subsection{Computation of a horizontal basis}
\label{ss:horizontal}

Using the previous algorithm
we can compute a diagonal basis of~$E[ℓ^{h+1}]$.
By Proposition~\ref{prop:diagonal-horizontal},
this gives us a horizontal basis of~$E[ℓ]$.
Thanks to Proposition~\ref{prop:push-horizontal},
we can use this information to improve horizontal points of~$E[ℓ^i]$
into horizontal points of~$E[ℓ^{i+1}]$, as illustrated in
Algorithm~\ref{alg:horizontal}.

\begin{algorithm}
\caption{\label{alg:horizontal}Computing a horizontal point of order~$ℓ^k$}
\begin{algorithmic}[1]
\REQUIRE $(P_0, Q_0)$: a diagonal basis of~$E[ℓ^{h+1}]$; $k$: an integer,
$k ≥ h + 1$.\\
\ENSURE $R$: a horizontal point of~$E[ℓ^k]$ with direction~$λ$.
\FOR{$i = 1$ to~$k-1$}
\STATE $ϕ_i \gets $ isogeny with kernel~$\chev{ℓ^{h} P_{i-1}}$
\STATE $Q_{i} \gets ϕ_i(Q_{i-1})$
\STATE\label{alg:horizontal:divide} $P' \gets \sfdiv(\ell, ϕ_i(P_{i-1}))$.
\STATE\label{alg:horizontal:frob} Write~$π(P') = λ P' + b Q_i$ for~$b ∈ ℤ/ℓℤ$ and
let $P_{i} \gets P' - (b/μ) Q_i$.
\ENDFOR
\RETURN\label{alg:horizontal:final} $R = \widehat{ϕ}_1 ∘ … ∘ \widehat{ϕ}_{k-1}
  (\sfdiv( ℓ^{k-(h+1)}, P_{k-1}) )$. 
\end{algorithmic}
\end{algorithm}
%\par\kern -4ex
\begin{prop}\label{th:horizontal}
  Algorithm~\ref{alg:horizontal} is correct and computes its output
  using an expected $O(\RR(k-\beta) + k\RR(h-β+1) + kℓ^2\MM(ℓ^{h-β+1}))$
  operations in $\F_q$.
\end{prop}
\begin{proof}
Let $E_i$ be the image curve of $ϕ_i$.
We check that at step~$i$ of the loop,
the points~$(P_i, Q_i)$ form a diagonal basis of~$E_i[ℓ^{h+1}]$,
and $ϕ_i$~has direction~$λ$.
The fact that $R$~is horizontal is then a consequence
of Proposition~\ref{prop:push-horizontal}.
The two most expensive operations in the loop are
Steps~\ref{alg:horizontal:divide} and~\ref{alg:horizontal:frob},
costing respectively $O(\RR(h-β+1))$ and $O(ℓ^2\MM(ℓ^{h-β+1}))$, as
discussed in the proof of Proposition~\ref{th:diagonal}. They are
repeated $k$ times. Finally, Step~\ref{alg:horizontal:final} is
dominated by the last $\sfdiv$ operation, which costs $O(\RR(k-β))$.
\end{proof}

One application of Algorithm~\ref{alg:diagonal} (with input~$k ← h+1$)
and two applications of Algorithm~\ref{alg:horizontal} allow us
to compute a horizontal basis of~$E[ℓ^k]$.
This could be done directly with Algorithm~\ref{alg:diagonal} instead,
but that would require computing in an extension $F_{k+h-\beta}$.

%%%%%%%%%%%%%%%

\section{Interpolation step}
\label{sec:interpolation}

After constructing bases $(P,Q)$ of $E[ℓ^k]$ and $(P',Q')$ of
$E'[ℓ^k]$ using the algorithms of the previous section, our algorithm
computes the polynomial with coefficients in $\F_q$ mapping $x(P)$ to
$x(P')$, $x(Q)$ to $x(Q')$, and the other abscissas accordingly.  In
this section we give an efficient algorithm for this specific
interpolation problem. The algorithm appeared in~\cite{df10} in the
context of the Artin-Schreier extensions used in Couveignes' isogeny
algorithm; it uses original ideas from~\cite{enge+morain03}. We recall
this algorithm here, and adapt the complexity analysis to our setting 
of Kummer extensions.

%% \subsection{Rational interpolation}

We start by tackling a simpler problem. We suppose we have constructed
a tower of Kummer extensions $\F_q=F_0⊂F_1⊂\cdots⊂F_n$, with
$[F₁:F₀]\mid(ℓ-1)$, and $[F_{i+1}:F_i]=ℓ$ for any $i>0$. Given two
elements $v,w∈F_n\setminus F_{n-1}$, we want to compute polynomials
$T$ and $L$ such that:
\begin{itemize}
\item $T \in \F_q[x]$ is the minimal polynomial of $v$, of degree
  $d=\deg T<ℓ^n$;
\item $L$ is in $\F_q[x]$, of degree less than $d$, and $L(v)=w$.
\end{itemize}
Observe that, since $v,w∉F_{n-1}$, we necessarily have $v_ℓ(d)=n-1$,
so that $ℓ^{n-1}≤d<ℓ^n$.
Using a fast interpolation algorithm~\cite[Chapter~10.2]{vzGG}, the
polynomials $T$ and $L$ could be computed in
$O\bigl(n\MM(ℓ^{2n})\log(ℓ)\bigr)$ operations in $\F_q$. We can do
much better by exploiting the form of the Kummer tower, and the
Frobenius algorithm given in Lemma~\ref{lemma:frob-ell}.

Following~\cite{df10}, we first compute $T$, starting from
$T^{(0)}=x-v$.  We let $\sigma_i$ be the map that takes all the
coefficients of a polynomial in $F_{n-i}[x]$ to the power
$\#F_{n-i-1}$. For $i=0,\dots,n-1$, suppose we know a polynomial
$T^{(i)}$ of degree $\ell^i$ in $F_{n-i}[x]$. Then, compute the
polynomials $T^{(i,j)}$ given by
$T^{(i,j)}= \sigma_i^j\bigl (T^{(i)} \bigr)$
for $0 \le j \le \ell-1$,
and define
\begin{equation}
  \label{eq:interp}
  T^{(i+1)}=\prod_{j=0}^{b} T^{(i,j)}
  \qquad\text{with}\qquad
  b = \begin{cases}
    ℓ-1 &\text{if $i<n-1$,}\\
    d/ℓ^{n-1} &\text{otherwise.}
  \end{cases}
\end{equation}
One easily sees that
$T^{(i+1)}$ is the minimal polynomial of $v$ over $F_{n-i+1}$.

\begin{lem}\label{lemma:interpolation:minpoly}
  The cost of computing $T$ is $O(n\MM(ℓ^{n+1})\log(ℓ))$
  operations in $\F_q$.
\end{lem}

\begin{proof}
  At each step $i$, from the knowledge of $T^{(i)}$ we compute all
  $T^{(i,j)}$ using Lemma~\ref{lemma:frob-ell}. The cost for a single
  polynomial $T^{(i,j)}$ is of $O(ℓ^iℓ^{n-i-1}\MM(ℓ))$ operations,
  \emph{i.e.} $O(ℓ^n\MM(ℓ))$ for all $O(\ell)$ of them.
  From the $T^{(i,j)}$'s we compute $T^{(i+1)}$ using a subproduct
  tree, as in~\cite[Lemma~10.4]{vzGG}. The result has degree
  $O(ℓ^{i+1})$ and coefficients in $F_{n-i}$, thus the overall cost is
  $O(\MM(ℓ^{n+1})\log(ℓ))$. After $T^{(i+1)}$ is computed this way, we
  can convert its coefficients to $F_{n-i-1}$ at no algebraic cost.
  Summing over all $i$, we obtain the stated complexity.
\end{proof}

We can finally proceed with the interpolation itself. First, compute
$w' = w/T'(v)$ and let $L^{(0)}=w'$.  Next, for $i=0,\dots,n-2$,
suppose we know a polynomial $L^{(i)}$ in $F_{n-i}[x]$ of degree less
than $\ell^i$. We compute the polynomials $L^{(i,j)}$ given by
$L^{(i,j)}= \sigma_i^j\bigl(L^{(i)}\bigr)$ and
\[L^{(i+1)} = \sum_{j=0}^{b} L^{(i,j)}\frac{T^{(i+1)}}{T^{(i,j)}},
  \qquad\text{$b$ defined as in Eq.~\eqref{eq:interp}}.\]
As shown in~\cite{df10}, $L^{(n)}$ is the
polynomial $L$ we are looking for.

\begin{prop}
  Given $v,w∈F_n\setminus F_{n-1}$, the cost of computing the
  minimal polynomial $T∈\F_q[x]$ of $v$ and the interpolating
  polynomial $L∈\F_q[x]$ such that $L(v)=w$ is $O(n\MM(ℓ^{n+1})\log(ℓ))$
  operations in $\F_q$.
\end{prop}
\begin{proof}
  After the polynomials $T^{(i)}$ have been computed, we need to
  compute $T'(v)$. This is done by means of successive Euclidean
  remainders, since
  $T'(v) = (((T' \bmod T^{(1)}) \bmod T^{(2)}) \cdots \bmod T^{(n)})$.
  At stage $i$, we have to compute the Euclidean division of a
  polynomial of degree $O(ℓ^{n-i+1})$ by one of degree $O(ℓ^{n-i})$ in
  $F_i[x]$. Using the complexities from Section~\ref{sub:towers} we
  deduce that each division can be done in time $O(\MM(ℓ^{n+1}))$, for
  a total of $O(n\MM(ℓ^{n+1}))$ operations. Then, computing
  $w' = w/T'(v)$ takes $O(\MM(\ell^n)\log(\ell^n))$ operations.

  Finally, at each step $i$, the polynomials $L^{(i,j)}$ are computed
  at a cost of $O(ℓ^n\MM(ℓ))$, as in the proof of
  Lemma~\ref{lemma:interpolation:minpoly}.  The computation of
  $L^{(i+1)}$ uses the same subproduct tree as for the computation of
  $T^{(i)}$, requiring $O(\log ℓ)$ additions, multiplications and
  divisions of polynomials of degree $O(ℓ^{i+1})$ with coefficients in
  $F_{n-i}$, for a total of $O(\MM(ℓ^{n+1})\log(ℓ))$. Summing over all
  $i$, the complexity statement follows readily.
\end{proof}

We end with the general problem of interpolating a polynomial in
$\F_q[x]$ at points of $F_n$.

\begin{prop}\label{prop:interpol}
  Let $(v_1,w_1),\dots,(v_s,w_s)$ be pairs of elements of $F_n$, let
  $t_i$ be the degree of the minimal polynomial of $v_i$, and let
  $t=\sum t_i$. The polynomials
  \begin{itemize}
  \item $T∈\F_q[x]$ of degree $t$ such that $T(v_i)=0$ for all $i$,
    and
  \item $L∈\F_q[x]$ of degree less than $t$ such that $L(v_i)=w_i$ for
    all $i$
  \end{itemize}
  can be computed using
  $O\bigl(\MM(t)\log(s) + n\MM(ℓ^2t)\log(ℓ)\bigr)$ operations in $\F_q$.
\end{prop}
\begin{proof}
  The polynomial $T$ is simply the product of all the minimal
  polynomials $T_i$. Let $n_i=v_ℓ(t_i)$, so that
  $v_i,w_i∈F_{n_i+1}\setminus F_{n_i}$, and $ℓ^{n_i}≤t_i<ℓ^{n_i+1}$.
  We convert $(v_i,w_i)$ to a pair of elements of $F_{n_i+1}$ at no
  algebraic cost, then we compute $T_i$ as explained previously at a
  cost of $O(n\MM(ℓ^{n_i+2})\log(ℓ))$ operations. Bounding $ℓ^{n_i}$ by
  $t_i$, summing over all $i$, and using the superlinearity of $\MM$,
  we obtain a total cost of $O(n\MM(ℓ^2t)\log(ℓ))$ operations.
  Simultaneously, we compute all the polynomials $L_i$ such that
  $L_i(v_i)=w_i$, at the same cost.

  Then we arrange the $T_i$'s into a binary subproduct tree and
  multiply them together. A balanced binary
  tree, though not necessarily optimal, has a depth of
  $O(\log (s))$, and requires $O(\MM(t))$ operations per level. Thus
  we can bound the cost of computing $T$ by $O(\MM(t)\log(s))$.

  Finally, using the same subproduct tree structure, we apply the
  Chinese remainder algorithm of~\cite[Chapter~10]{vzGG} to compute
  the polynomial $L$ at the same cost $O(\MM(t)\log(s))$.
\end{proof}

\section{The complete algorithm}
\label{sec:complete-algorithm}

We finally come to the description of the full algorithm. Given two
$j$-invariants, defining two elliptic curves $E$ and $E'$, and an
integer $r$, we want to compute an isogeny $ψ:E→E'$ of degree $r$.
Since the algorithms of Section~\ref{sec:acti-frob-endm} apply to
curves on top of volcanoes with cyclic crater, we first need to
determine a small Elkies prime $ℓ$ for $E$ and $E'$, and then reduce
to an explicit isogeny problem on the crater of the
$ℓ$-volcanoes. These steps are discussed and analyzed next.

%%%%%%%%%%%%%%%
\subsection{Finding a suitable $ℓ$-volcano}
\label{sub:shape-volcano}

Our algorithm uses an Elkies prime~$ℓ$.  Since~$d_K$ is not assumed to
be known yet, we need to be able to compute the height $h$ of the
volcano, the shape of its crater, as well as the shortest $ℓ$-isogeny
chain from~$E$ to the crater.

The algorithms of Fouquet and Morain~\cite{volcano} compute the height
$h$ and find a curve $E_{\max}$ on the crater at the cost of $O(ℓh^2)$
factorizations of the $ℓ$-th modular polynomial $Φ_ℓ$. The polynomial
$Φ_ℓ$ is computed using $\tildO(ℓ³\log(ℓ))$ boolean operations, then
each factorization costs an expected $O(\MM(ℓ)\log(ℓ)\log(ℓq))$
operations using the Cantor-Zassenhaus algorithm (more efficient
methods for special instances of volcanoes are presented
in~\cite{MiretMRV05} and in~\cite{ionica+joux13}, but we do not
discuss them). Working on $E$ and $E'$, we compute the shortest path
of $ℓ$-isogenies~$α: E → E_{\max}$, $α': E' → E'_{\max}$ linking the
curves~$E, E'$ to the craters.
We still have to determine the
shape of these craters.  Since the height~$h$ of the volcano is known,
using Algorithm~\ref{alg:diagonal} we can compute a matrix of
$π|E_{\max}[ℓ^{h+1}]$.  If this matrix has two distinct eigenvalues then the
crater is cyclic, otherwise it is not.

%% Note that these computations do not depend on the isogeny
%% degree $r$ and can be reused to compute many isogenies on $E$.

By Proposition~\ref{prop:parallel}, the depth of~$E$ and $E'$ below
their respective craters is the same.  By
Proposition~\ref{prop:parallel}~\ref{prop:parallel:ascent}, the
curves~$E_{\max}$ and~$E'_{\max}$ are again $r$-isogenous; we can use
our algorithm to compute such an isogeny~$ψ_{\max}$.  Then, 
since $ℓ$ is coprime to $r$, $ψ =
(α')^{-1} ∘ ψ_{\max} ∘ α$ is well defined and is the required $r$-isogeny.
Its kernel can
be computed in $O(h\MM(ℓr)\log(ℓr))$ operations by evaluating the dual isogeny $\hat{α}$
on the kernel of $ψ_{\max}$ via a sequence of resultants.

\subsection{Interpolating the isogeny}
\label{sub:final-interp}

We now assume that both curves~$E, E'$
have $ℓ$-maximal endomorphism rings.
We fix bases of~$E[ℓ^k]$, $E'[ℓ^k]$ and write~$π, π'$ for the matrices
of the Frobenius.
Since $ψ$~is rational, its matrix satisfies the relation~$π' · ψ = ψ · π$
in $ℤ_ℓ^{2×2}$ and hence in~$(ℤ/ℓ^k ℤ)^{2 × 2}$.

If diagonal bases of~$E[ℓ^k], E'[ℓ^k]$ are used, then,
since $π$~is a cyclic endomorphism of~$ℤ_ℓ^2$,
this condition seems to ensure that $ψ$~is a diagonal matrix;
however, $ℤ/ℓ^k ℤ$~is not an integral domain
and $π$~is congruent, modulo~$ℓ^h$, to the scalar matrix~$λ$,
so we can only say that~$ψ\pmod{ℓ^{k-h}}$ is diagonal.
If on the other hand we choose \emph{horizontal} bases
of~$E[ℓ^k], E'[ℓ^k]$ then, by Proposition~\ref{prop:parallel}~\ref{prop:parallel:func},
we know that $ψ$~is a diagonal matrix.

We then enumerate all the $ℓ^{2k-2}$ invertible diagonal matrices; for
each matrix~$M$, we interpolate the action of~$M$ on~$E[ℓ^k]$ as a
rational fraction, and verify that it is an $r$-isogeny. The
successful interpolation will be our explicit isogeny~$ψ$.
Precisely, we interpolate using the abscissas of non-zero points of $E[ℓ^k]$;
there are $(ℓ^{2k}-1)/2$ distinct such abscissas (or $2^{2k-1}+1$ when
$ℓ=2$).  The isogeny~$ψ$ acts on abscissas as a rational fraction of
degrees~$(r, r-1)$, which is thus defined by $2r$ coefficients; 
knowing this rational function allows us to find the kernel of $\psi$,
and recover $\psi$ itself using V\'elu's formulas. For this
method to work, we therefore select the smallest~$k ≥ h+1$ such
that~$ℓ^{2k}-1 > 4r$.

Summarizing, our algorithm for two
$ℓ$-maximal curves proceeds as follows:
\begin{enumerate}
\item\label{alg:ours:horizontal} Use Algorithms~\ref{alg:diagonal}
  and~\ref{alg:horizontal} to compute horizontal bases $(P,Q), (P', Q')$
  of $E[ℓ^k], E'[ℓ^k]$;
\item\label{alg:ours:T} Compute the polynomial~$T$ vanishing
  on the abscissas of $\langle P,Q\rangle$ as in
  Section~\ref{sec:interpolation};
\item\label{alg:ours:for} For each invertible diagonal matrix
  $\smat{a&0\\0&b}$ in $(ℤ/ℓ^k ℤ)^{2×2}$:
  \begin{enumerate}
  \item\label{alg:ours:interp} compute the interpolation polynomial
    $L_{a,b}$ such that
    $L_{a,b} (x (u P + v Q)) = x(a\, u\,P' + b\,v\, Q')$ for all
    $u, v ∈ ℤ/ℓ^k ℤ$;
  \item\label{alg:ours:cauchy} Use the \emph{Cauchy interpolation
      algorithm} of~\cite[Chapter~5.8]{vzGG} to compute a rational
    fraction $F_{a,b}≡L_{a,b}\pmod{T}$ of degrees~$(r, r-1)$;
  \item If $F_{a,b}$ defines an isogeny of degree $r$, return it and
    stop.
  \end{enumerate}
\end{enumerate}

\begin{prop}
  \label{prop:full-complexity}
  Assuming that $ℓ^h<√ r$, the algorithm above
  computes an isogeny ${ψ:E→E'}$ in expected time
$O\Bigl(\bigl(r ℓ^2\MM(rℓ^4) +  \MM(rℓ^3)\log(ℓq)\bigr)\log(r)\log(ℓ) \Bigr).$
\end{prop}
\begin{proof}
By definition of~$k$, we know that~$ℓ^{2k} ∈ O(rℓ^2)$.
  By Proposition~\ref{prop:classes}, there is a $β<h$ such that
  $E[ℓ^k]$ is contained in $E(F_n)$ with $n=k-β$. We thus construct
  the Kummer tower $F₀⊂\cdots⊂F_n$, and we do the precomputations
  required by Lemma~\ref{lemma:frob-ell} at a cost of
  $O(ℓ\MM(ℓ)\log(q))$.

  Bounding $h$ by $k-1$, Step~\ref{alg:ours:horizontal} costs on average
  $O(k\RR(k-\beta) + kℓ^2\MM(ℓ\sqrt{r}) + ℓ\MM(ℓ^2)\log(ℓ)\log(ℓq))$ according to
  Propositions~\ref{th:diagonal} and~\ref{th:horizontal}.
  Using the most pessimistic estimates of
  Section~\ref{sub:towers}, we see that this cost is bounded by
  $O(\MM(rℓ^3)\log(r)\log(ℓ)\log(ℓq))$.

  By Proposition~\ref{prop:classes}~\ref{prop:classes:count}, there
  are at most~$O(k· ℓ^{k+β})$ Galois classes in~$E[ℓ^k]$. In order to
  apply the algorithms of Section~\ref{sec:interpolation}, we need to
  compute a representative for each class. Each representative is
  computed from the basis $(P,Q)$ using point multiplication by two
  scalars $≤ℓ^k$ in the field $F_n$, which costs
  $O(\MM(ℓ^n)\log(ℓ^k))$ operations. We thus have a total cost of
  $O(k\MM(ℓ^{2k})\log(ℓ^k)) ⊂ O(\MM(rℓ^2)\log(r)^2)$ to compute all such representatives.

  Then, using Proposition~\ref{prop:interpol}, where the total degree
  is~$t = (ℓ^{2k}-1)/2∈O(r ℓ^2)$, and the number of interpolation points
  is $s∈O(k·ℓ^{k+β})$, we can compute the polynomials~$T$
  and~$L_{a,b}$ at a cost of~$O(\MM(r ℓ^4)\log(r)\log(ℓ))$.  The cost
  of computing $F_{a,b}$, and identifying the isogeny, is dominated by
  that of computing~$L_{a,b}$~\cite[§3.3]{df10}.  Finally, in general
  approximately ${ℓ^{2k}=O(r ℓ^2)}$~candidate matrices must be tried before
  finding the isogeny.
\end{proof}

\subsection{Overall complexity}
\label{sub:complexity}

By a result of Shparlinski and
Sutherland~\cite[Theorem~1]{shparlinski2014distribution},
for almost all primes $q$ and curves $E/\F_q$, for $L≥\log(q)^ε$ for any $ε>0$,
asymptotically half of the primes $ℓ ≤ L$ are Elkies primes.
Hence, we expect to have enough small Elkies primes to apply our algorithm.
The following theorem states a worst case bound depending on~$r$ and~$q$ alone.

\begin{thm*}
  For almost all primes $q$ and curves $E,E'$ over $\F_q$, it is
  possible to solve the “Explicit Isogeny Problem” in expected time
  $O\bigl(r \MM(r\log(q)^6) \log(r) \loglog(q) \bigr)$.
\end{thm*}

\begin{proof}
  Given a curve $E$, we search for the smallest Elkies prime
  satisfying the conditions of Proposition~\ref{prop:full-complexity}.
  As a special case of~\cite[Theorem~1]{shparlinski2014distribution},
  we can take $L∈O(\log(q))$ such that the product of all
  Elkies primes $ℓ≤L$ exceeds $Ω(√q)$.
On the other hand, we discard those primes~$ℓ ≤ L$ for which
the height~$h$ satisfies~$ℓ^{h} > √r$;
since those discarded primes are divisors of~$√{d_K}$,
their product is at most~$O(√q)$.
This shows that there remains enough “good” Elkies primes
in~$\llbracket 1,L\rrbracket$,
so that in the worst case $ℓ ∈ O(\log(q))$.
%   For any such~$ℓ$, the
%   corresponding height~$h = h_ℓ$
%   is~$h_ℓ = \frac 12 v_ℓ(d_K)$, so that $\prod ℓ^{h_ℓ}≤2√q$.
%   Hence, even discarding those primes such that~$ℓ^{h_ℓ} > √r$,
%   there will be enough good Elkies primes left in the interval $\llbracket 1,L\rrbracket$.
%   Therefore in the worst case~$ℓ = O(\log(q))$.

  The most expensive steps in Section~\ref{sub:shape-volcano} are
  the computation and the factorization of the modular polynomials
  for all primes up to $ℓ$.
  This is well within~$O(\log(q)^6)$.
  The stated complexity follows then from
  substituting~$ℓ = O(\log(q))$
  in Proposition~\ref{prop:full-complexity}.
\end{proof}

\section{Conclusion and experimental results}
\label{sec:implem}

In the previous sections we have obtained a Las Vegas algorithm with
an interesting asymptotic complexity.  In particular, in the favorable
case where~$ℓ = O(1)$, the running time of the algorithm is
quasi-quadratic in the isogeny degree~$r$ and quasi-linear in~$\log
q$.  Thus we expect it to be practical, and a substantial improvement
over Couveignes' original algorithm, at least when small parameters
$ℓ$ and $h$ can be found quickly. A large $ℓ$ or $h$ adversely affects
performance in the following ways:
\begin{itemize}
\item All modular polynomials up to $Φ_ℓ$ must be computed or
  retrieved from tables.
\item All degrees $(ℓ^{2k}-1)/4≤r<(ℓ^{2k+1}-1)/4$ require essentially
  the same computational effort, thus resulting in a \emph{staircase
    behavior} when $r$ increases.
\item Because we must have $k>h$, all degrees $r$ smaller than
  $(ℓ^{2h+2}-1)/4$ require the same computational effort.
\end{itemize}
For these reasons, it is wisest in practice to set small \emph{a
  priori} bounds on $ℓ$ and $h$, and only run our algorithm when
parameters within these bounds can be found.

To validate our findings, we implemented a simplified version of our main algorithm using
SageMath v7.1~\cite{sage}. In our current implementation, we only
handle the case $\ell=2$ and we work only with curves on the
crater of a $2$-volcano.  We implemented the
construction of Kummer towers described in~\cite{DoSc12}, in the
favorable case where $p = 1 \bmod 4$. Source code and benchmark data
are available in the GitHub project
\url{https://github.com/Hugounenq-Cyril/Two_curves_on_a_volcano/}.

\begin{figure}%\label{fig:p=101}
\centering
\includegraphics[width=0.49\textwidth]{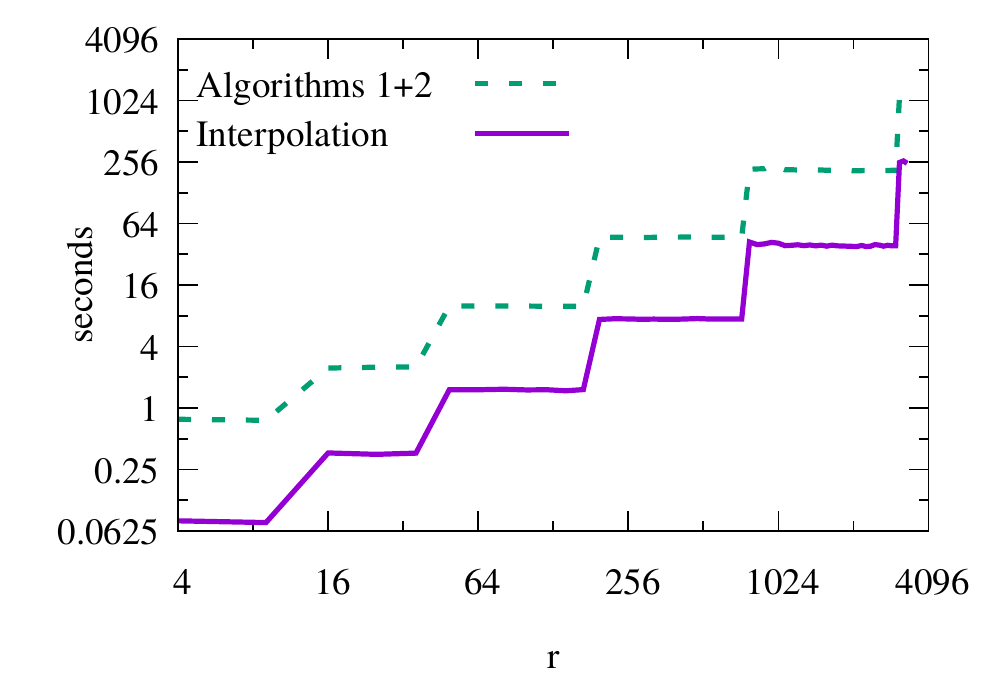}
\hfill
\includegraphics[width=0.49\textwidth]{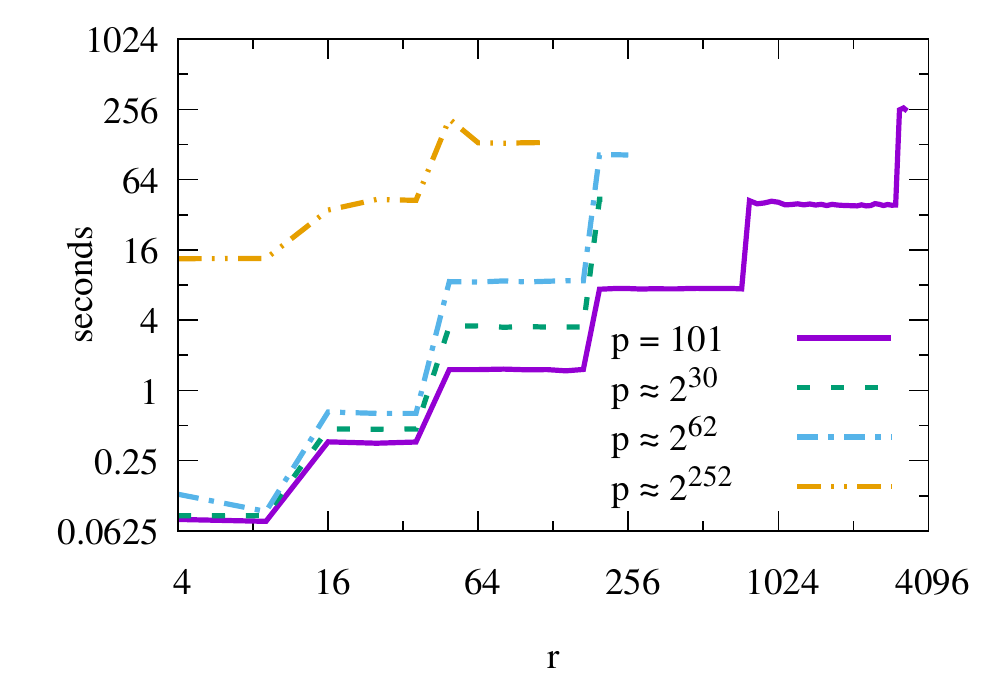}
\caption{Left: comparison of horizontal basis computation and
  interpolation phases, for a fixed curve defined over $\F_{101}$, and
  increasing $r$. Right: Comparison of one interpolation phase for
  $\F_{101}$, $\F_{2^{30}+669}$, $\F_{2^{62}+189}$ and $\F_{2^{252}+421}$, and increasing
  $r$. Plots in logarithmic scale.}
\label{fig:benchs}
\end{figure}

We ran benchmarks on an Intel Xeon E5530 CPU clocked at 2.4GHz. We
fixed a base field $\F_q$ and an elliptic curve $E$ with height $h=3$ and $\beta=2$,
then ran our algorithm to compute the multiplication-by-$r$ isogeny
$E→E$, for $r$ increasing.  The torsion levels involved in the
computations varied from $2^3$ to $2^8$.  Figure~\ref{fig:benchs}
(left) shows the running times for the computation of the horizontal
basis of $E[ℓ^k]$, and for one execution of the interpolation
step. Running times are close to linear in $r$,
as expected. The staircase behavior of our algorithm is apparent from
the plot. Since the interpolation steps must be repeated $\sim r$ times,
we focus on this step to compare the running
time for different base fields. In Figure~\ref{fig:benchs} (right)
we observe that the dependency in $q$, although much better than in
Couveignes' original algorithm, is higher than what the theoretical 
analysis would predict. This may due to low-level implementation details
of SageMath, which, in the current implementation, are beyond our control.

In conclusion, our algorithm shows promise of being of practical
interest within selected parameter ranges. Generalizing it to work
with Atkin primes would considerably enlarge its applicability range;
we hope to develop such a generalization in a future work. On the
practical side, we plan to work on two improvements that seem within
reach. First, the reduction from generic curves to $ℓ$-maximal curves
seems superfluous and unduly expensive: it would be interesting to
generalize the concept of horizontal bases to any curve. Second, a
multi-modular approach interpolating on a torsion group of composite
order is certainly possible, and could improve the running time of our
algorithm by allowing it to work in smaller extension fields.

\begin{acknowledgments}
  We would like to thank the anonymous referees for their careful
  review and their insightful remarks.
\end{acknowledgments}

\bibliographystyle{plain}
\bibliography{refs}

\begin{thebibliography}{10}

\bibitem{atkin88}
Arthur O.~L. Atkin.
\newblock The number of points on an elliptic curve modulo a prime.
\newblock Mail to the nmbrthry mailing list, 1988.

\bibitem{atkin91}
Arthur O.~L. Atkin.
\newblock The number of points on an elliptic curve modulo a prime.
\newblock Mail to the nmbrthry mailing list, 1991.

\bibitem{Bostan}
Alin Bostan, Fran\c{c}ois Morain, Bruno Salvy, and {\'{E}}ric Schost.
\newblock Fast algorithms for computing isogenies between elliptic curves.
\newblock {\em Mathematics of Computation}, 77(263), 2008.

\bibitem{sutherland10:modpol}
Reinier Br{\"o}ker, Kristin Lauter, and Andrew Sutherland.
\newblock Modular polynomials via isogeny volcanoes.
\newblock {\em Mathematics of Computation}, 81(278):1201--1231, 2012.

\bibitem{charlap1991enumeration}
Leonard~S Charlap, Raymond Coley, and David~P Robbins.
\newblock Enumeration of rational points on elliptic curves over finite fields,
  1991.
\newblock Preprint.

\bibitem{charles+lauter+goren09}
Denis~X. Charles, Kristin~E. Lauter, and Eyal~Z. Goren.
\newblock Cryptographic hash functions from expander graphs.
\newblock {\em Journal of Cryptology}, 22(1):93--113, January 2009.

\bibitem{couveignes94}
Jean-Marc Couveignes.
\newblock {\em {Quelques calculs en th{\'{e}}orie des nombres}}.
\newblock PhD thesis, Universit\'{e} de Bordeaux, 1994.

\bibitem{couveignes96}
Jean-Marc Couveignes.
\newblock Computing {l-Isogenies} using the {p-Torsion}.
\newblock In {\em ANTS-II: Proceedings of the Second International Symposium on
  Algorithmic Number Theory}, pages 59--65, London, UK, 1996. Springer-Verlag.

\bibitem{couveignes00}
Jean-Marc Couveignes.
\newblock Isomorphisms between {A}rtin-{S}chreier towers.
\newblock {\em Mathematics of Computation}, 69(232):1625--1631, 2000.

\bibitem{df10}
Luca De~Feo.
\newblock Fast algorithms for computing isogenies between ordinary elliptic
  curves in small characteristic.
\newblock {\em Journal of Number Theory}, 131(5):873--893, May 2011.

\bibitem{DeDoSc13}
Luca {De Feo}, Javad Doliskani, and {\'E}ric Schost.
\newblock Fast algorithms for $\ell$-adic towers over finite fields.
\newblock In {\em ISSAC'13: {P}roceedings of the 2013 international symposium
  on Symbolic and algebraic computation}, pages 165--172. ACM, 2013.

\bibitem{defeo+jao+plut12}
Luca De~Feo, David Jao, and J{\'e}r{\^o}me Pl{\^u}t.
\newblock Towards quantum-resistant cryptosystems from supersingular elliptic
  curve isogenies.
\newblock {\em Journal of Mathematical Cryptology}, 8(3):209--247, 2014.

\bibitem{df+schost12}
Luca De~Feo and {\'E}ric Schost.
\newblock Fast arithmetics in {A}rtin-{S}chreier towers over finite fields.
\newblock {\em Journal of Symbolic Computation}, 47(7):771--792, 2012.

\bibitem{DoSc12}
Javad Doliskani and \'Eric Schost.
\newblock Computing in degree $2^k$-extensions of finite fields of odd
  characteristic.
\newblock {\em Designs, Codes and Cryptography}, 74(3):559--569, 2015.

\bibitem{elkies98}
Noam~D. Elkies.
\newblock Elliptic and modular curves over finite fields and related
  computational issues.
\newblock In {\em Computational perspectives on number theory (Chicago, IL,
  1995)}, volume~7 of {\em Studies in Advanced Mathematics}, pages 21--76,
  Providence, RI, 1998. AMS International Press.

\bibitem{enge+morain03}
Andreas Enge and Fran\c{c}ois Morain.
\newblock Fast decomposition of polynomials with known galois group.
\newblock In {\em AAECC'03: Proceedings of the 15th international conference on
  Applied algebra, algebraic algorithms and error-correcting codes}, pages
  254--264, Berlin, Heidelberg, 2003. Springer-Verlag.

\bibitem{volcano}
Mireille Fouquet and Fran\c{c}ois Morain.
\newblock Isogeny volcanoes and the {SEA} algorithm.
\newblock In {\em Algorithmic number theory ({S}ydney, 2002)}, volume 2369 of
  {\em Lecture Notes in Comput. Sci.} Springer, Berlin, 2002.

\bibitem{gallant+lambert+vanstone01}
Robert~P. Gallant, Robert~J. Lambert, and Scott~A. Vanstone.
\newblock Faster point multiplication on elliptic curves with efficient
  endomorphisms.
\newblock In {\em CRYPTO '01: Proceedings of the 21st Annual International
  Cryptology Conference on Advances in Cryptology}, pages 190--200, London, UK,
  2001. Springer-Verlag.

\bibitem{ionica+joux13}
Sorina Ionica and Antoine Joux.
\newblock Pairing the volcano.
\newblock {\em Mathematics of Computation}, 82(281):581--603, 2013.

\bibitem{jao+soukharev2014-signatures}
David Jao and Vladimir Soukharev.
\newblock Isogeny-based quantum-resistant undeniable signatures.
\newblock In {\em Post-Quantum Cryptography: 6th International Workshop,
  PQCrypto 2014}, pages 160--179, Waterloo, ON, Canada, 2014. Springer
  International Publishing.

\bibitem{kaltofen+shoup97}
Erich Kaltofen and Victor Shoup.
\newblock Fast polynomial factorization over high algebraic extensions of
  finite fields.
\newblock In {\em ISSAC '97: Proceedings of the 1997 International Symposium on
  Symbolic and Algebraic Computation}, pages 184--188, New York, NY, USA, 1997.
  ACM.

\bibitem{kohel}
David Kohel.
\newblock {\em Endomorphism rings of elliptic curves over finite fields}.
\newblock PhD thesis, University of California at Berkeley, 1996.

\bibitem{1602.00244}
Pierre Lairez and Tristan Vaccon.
\newblock On p-adic differential equations with separation of variables.
\newblock Preprint available at \url{http://arxiv.org/abs/1602.00244}, 2016.

\bibitem{lercier+sirvent08}
Reynald Lercier and Thomas Sirvent.
\newblock On {E}lkies subgroups of \(\ell\)-torsion points in elliptic curves
  defined over a finite field.
\newblock {\em Journal de th\'{e}orie des nombres de Bordeaux}, 20(3):783--797,
  2008.

\bibitem{longa+sica14}
Patrick Longa and Francesco Sica.
\newblock Four-dimensional {G}allant--{L}ambert--{V}anstone scalar
  multiplication.
\newblock {\em Journal of Cryptology}, 27(2):248--283, 2014.

\bibitem{mauer+menezes+teske01}
Markus Maurer, Alfred Menezes, and Edlyn Teske.
\newblock Analysis of the {GHS} {W}eil descent attack on the {ECDLP} over
  characteristic two finite fields of composite degree.
\newblock In {\em INDOCRYPT '01: Proceedings of the Second International
  Conference on Cryptology in India}, pages 195--213. Springer-Verlag, 2001.

\bibitem{MiretMRV05}
Josep~M. Miret, Ramiro Moreno, Ana Rio, and Magda Valls.
\newblock Determining the 2-sylow subgroup of an elliptic curve over a finite
  field.
\newblock {\em Mathematics of Computation}, 74(249):411--427, 2005.

\bibitem{schoof85}
Ren\'{e} Schoof.
\newblock Elliptic curves over finite fields and the computation of square
  roots mod \(p\).
\newblock {\em Mathematics of Computation}, 44(170):483--494, 1985.

\bibitem{schoof95}
Ren\'{e} Schoof.
\newblock Counting points on elliptic curves over finite fields.
\newblock {\em Journal de Th\'{e}orie des Nombres de Bordeaux}, 7(1):219--254,
  1995.

\bibitem{Serre.Arith}
Jean-Pierre Serre.
\newblock {\em {Cours d'arithmétique}}.
\newblock Presses Universitaires de France, 1970.

\bibitem{SL2}
Jean-Pierre Serre.
\newblock {\em Arbres, amalgames, $SL_2$}, volume~46 of {\em Astérisque}.
\newblock Société Mathématique de France, 1977.

\bibitem{shparlinski2014distribution}
Igor~E Shparlinski and Andrew~V Sutherland.
\newblock {On the distribution of {A}tkin and {E}lkies primes}.
\newblock {\em Foundations of Computational Mathematics}, 14(2):285--297, 2014.

\bibitem{Sil}
Joseph~H. Silverman.
\newblock {\em The arithmetic of elliptic curves}, volume 106 of {\em Graduate
  Texts in Mathematics}.
\newblock Springer-Verlag, New York, 1992.

\bibitem{Stol}
Anton Stolbunov.
\newblock Constructing public-key cryptographic schemes based on class group
  action on a set of isogenous elliptic curves.
\newblock {\em Advances in Mathematics of Communications}, 4(2), 2010.

\bibitem{sutherland2013isogeny}
Andrew Sutherland.
\newblock Isogeny volcanoes.
\newblock In {\em ANTS X: Proceedings of the Algorithmic Number Theory 10th
  International Symposium}, volume~1, pages 507--530. Mathematical Sciences
  Publishers, 2013.

\bibitem{sutherland2013evaluation}
Andrew Sutherland.
\newblock On the evaluation of modular polynomials.
\newblock In {\em ANTS X: Proceedings of the Algorithmic Number Theory 10th
  International Symposium}, volume~1, pages 531--555. Mathematical Sciences
  Publishers, 2013.

\bibitem{tate1966endomorphisms}
John Tate.
\newblock Endomorphisms of abelian varieties over finite fields.
\newblock {\em Inventiones mathematicae}, 2(2):134--144, 1966.

\bibitem{teske06}
Edlyn Teske.
\newblock An elliptic curve trapdoor system.
\newblock {\em Journal of Cryptology}, 19(1):115--133, January 2006.

\bibitem{sage}
{The Sage Developers}.
\newblock {\em {S}age {M}athematics {S}oftware ({V}ersion 7.0)}, 2016.

\bibitem{velu71}
Jean V{\'{e}}lu.
\newblock Isog{\'{e}}nies entre courbes elliptiques.
\newblock {\em Comptes Rendus de l'Acad\'{e}mie des Sciences de Paris},
  273:238--241, 1971.

\bibitem{vzGG}
Joachim von~zur Gathen and Jurgen Gerhard.
\newblock {\em Modern Computer Algebra}.
\newblock Cambridge University Press, New York, NY, USA, 1999.

\bibitem{vzgathen+shoup92}
Joachim von~zur Gathen and Victor Shoup.
\newblock Computing {F}robenius maps and factoring polynomials.
\newblock In {\em STOC '92: Proceedings of the twenty-fourth annual ACM
  symposium on Theory of computing}, pages 97--105, New York, NY, USA, 1992.
  ACM.

\end{thebibliography}

\appendix

\section{Galois classes in~$E[ℓ^k]$}
\label{ap:galois}

We give here the full decomposition of $E[ℓ^k]$ in Galois classes.
This is a more precise form of Proposition~\ref{prop:classes}~(v).
\begin{prop}\label{prop:orbites-l-torsion}
Let~$E$ be an elliptic curve with $ℓ$-maximal endomorphism ring.
Assume $ℓ ≠ 2$, $λ ≡ μ ≡ 1 \pmod{ℓ}$ and let~$α = v_ℓ(λ-1), β=v_ℓ(μ-1)$.
Write~$ν(x, y) = \min (x+y, x+β-1, y+α-1)$
and~$ρ(x, y) = x+y - ν(x, y) = \max (0, x-α+1, y-β+1)$.
The decomposition of the group~$E[ℓ^k]$ in Galois classes is as follows:
\begin{enumerate}
\item for~$i, j = 1, …, k-1$:
$(ℓ-1)^2 · ℓ^{ν(i,j)}$ classes of size~$ℓ^{ρ(i,j)}$;
\item for~$i = 1, …, k-1$:
$(ℓ-1) · ℓ^{\min (i, α-1)}$ classes of size~$ℓ^{\max (0, i-α+1)}$, and
$(ℓ-1) · ℓ^{\min (i, β-1)}$ classes of size~$ℓ^{\max (0, i-β+1)}$;
\item the $ℓ^2$ singleton classes of~$E[ℓ]$.
% \item \todo{le cas où $ℓ=2$}
\end{enumerate}
\end{prop}
\begin{proof}
Fix a basis~$(P, Q)$ of~$E[ℓ^k]$ such that~$π(P)=λP$, $π(Q)=μQ$.
Studying the Galois orbits of~$E[ℓ^k]$
means studying the map~$ℤ_ℓ^2 → ℤ_ℓ^2, (x, y) ↦ (λ x, μ y)$.
In other words, the orbits correspond to elements of~$ℤ_ℓ^2$
modulo the \emph{multiplicative} subgroup generated by~$(λ, μ)$.
An easy way to describe this
is to consider a \emph{multiplicative lattice} in~$(ℚ_ℓ^×)^2$.

Let~$ξ$ be a primitive $(ℓ-1)$-th root of unity in~$ℤ_ℓ$.
Then by~\cite[Théorème II.3.2]{Serre.Arith},
the map~$f(x, y, z) = ℓ^x· ξ^y· \exp (ℓ z)$
is a group isomorphism between~$ℤ × (ℤ/(ℓ-1) ℤ) × ℤ_ℓ$ and~$ℚ_ℓ^{×}$.
For~$i ∈ \bcro{0,k-1}$ and~$c ∈ ℤ/(ℓ-1)ℤ$,
let~$V(i,c)$ be the image in $ℤ/ℓ^k ℤ$ of the map~$f(k-1-i,c,–)$:
then the multiplicative structure of $V(i, c)$
is that of a principal homogeneous space under~$ℤ/ℓ^i ℤ$.
We also define~$W(i,j,c,d) = V(i, c) · P \,+\, V(j, d) · Q ⊂ E[ℓ^k]$.

Since~$λ ≡ 1 \pmod{ℓ}$, we may write~$λ = f(0,0,u\, ℓ^{α-1})$
and~$μ = f(0,0, v\, ℓ^{β-1})$
for some~$u, v ∈ ℤ_ℓ^{×}$.
This implies that the set~$W(i,j,c,d)$ is stable under Galois.
Moreover, the orbits of~$W(i,j,c,d)$ correspond bijectively to
points of a fundamental domain of the lattice~$Λ_{i,j}$ generated by
the columns of~$\smat{ℓ^i & 0 & u ℓ^{α-1}\\ 0 & ℓ^j & v ℓ^{β-1}}$,
whereas the size of each orbit is~$[(ℤ/ℓ^i ℤ)×(ℤ/ℓ^j ℤ)\::\: Λ_{i,j}]$.
By using elementary column manipulations,
we find that the covolume of~$Λ_{i,j}$ is~$ℓ^{ν(i,j)}$,
hence the point~(i) of the proposition.
(The case~$i = j = 0$ yields singleton classes in~$E[ℓ]$).

The union of all the sets~$W(j,i,c,d)$
is exactly the set of all~$x P + y Q$ for~$x, y ≠ 0$.
We obtain the classes of~(ii) by considering
the sets~$V(i, c) · P$ and~$V(j,d) · Q$.
\end{proof}

We now state the equivalent proposition when~$ℓ = 2$.
The proof is much the same as in the odd case.

\begin{prop}\label{prop:orbites-2-torsion}
Let~$E$ be an elliptic curve with $2$-maximal endomorphism ring.
Assume $λ ≡ μ ≡ 1 \pmod{4}$ and let~$α = v_2(λ-1), β=v_2(μ-1)$.
Write~$ν_2(x, y) = \min (x+y, x+β-2, y+α-2)$
and~$ρ_2(x, y) = x+y - ν_2(x, y) = \max (0, x-α+2, y-β+2)$.
The decomposition of the group~$E[2^k]$ in Galois classes is as follows:
\begin{enumerate}
\item for~$i, j = 1, …, k-2$:
$4 · 2^{ν_2(i,j)}$ classes of size~$2^{ρ_2(i,j)}$;
\item for~$i = 1, …, k-2$:
$4 · 2^{\min (i, α-2)}$ classes of size~$2^{\max (0, i-α+2)}$, and
$4 · 2^{\min (i, β-2)}$ classes of size~$2^{\max (0, i-β+2)}$.
\item the 16 singleton classes of~$E[4]$.
\end{enumerate}
\end{prop}
Note that if $λ$ or~$μ ≡ -1 \pmod{4}$ then
by replacing the base field by a quadratic extension,
we can always ensure that the condition $λ ≡ μ ≡ 1 \pmod{4}$ is
satisfied.
% \begin{proof}%<<<
% The proof is much the same as that of Prop.~\ref{prop:orbites-l-torsion}.
% The case~$ℓ = 2$ of~\cite[Théorème II.3.2]{Serre.Arith}
% states that $f(x,y,z) = 2^x · (-1)^y · \exp (4z)$
% is an isomorphism between~$ℤ × (ℤ/2ℤ) × ℤ_2$ and~$ℚ_2^×$.
% Let~$V(i, c)$ be the image modulo~$2^k$ of~$f(k-2-i, c, -)$;
% then $V(i, c)$ has cardinal~$2^{i}$ for~$i ≥ 0$,
% while~$V(-1, 0) = V(-1, 1)$ is the singleton $\acco{2^{k-1}}$.
% 
% Let~$P, Q$ be diagonal generators of~$E[2^{k}]$
% and define~$W(i,j,c,d) = V(i,c) P + V(j,d) Q$;
% also write~$λ = f(0,0,2^{α-2} u)$, $μ = f(0,0,2^{β-2} v)$.
% Then the Galois orbits in~$W(i,j,c,d)$ correspond to $ℤ_2^2$ modulo
% the lattice~$Λ_{i,j} = \smat{2^{i}&0&2^{α-2} u\\0&2^{j}&2^{β-2} v}$.
% We find that $Λ_{i,j}$ has covolume~$2^{ν_2(i,j)}$.
% This gives the orbits of point~(i) of the proposition,
% while the case~$i = j = 0$ gives singleton orbits.
% 
% By considering the orbits of~$V(i,c) P + ε Q$ and~$ε P + V(i,c) Q$
% for~$ε ∈ \acco{0, 2^{k-1}}$, we obtain the orbits
% of point~(ii).
% \end{proof}%>>>
%>>>1
\affiliationone{Luca De Feo\\
LMV -- UVSQ\\
45 avenue des États-Unis\\
78035 Versailles\\
France\\
ORCiD: \href{http://orcid.org/0000-0002-9321-0773}{0000-0002-9321-0773}}
\affiliationtwo{Jérôme Plût\\
ANSSI\\
51, boulevard de La Tour-Maubourg\\
75007 Paris\\
France}
\affiliationthree{Cyril Hugounenq\\
LMV -- UVSQ\\
45 avenue des États-Unis\\
78035 Versailles\\
France}
\affiliationfour{Éric Schost\\
Cheriton School of Computer Science\\
University of Waterloo\\
Canada}
\end{document}